\documentclass{amsart}
\usepackage{amsmath}
  \usepackage{paralist}
  \usepackage{pb-diagram}
  \usepackage{graphics} 
  \usepackage{epsfig} 
 \usepackage[colorlinks=true]{hyperref}
\hypersetup{urlcolor=blue, citecolor=red}

\newtheorem{theorem}{Theorem}[section]
\newtheorem{corollary}[theorem]{Corollary}

\newtheorem{lemma}[theorem]{Lemma}
\newtheorem{proposition}[theorem]{Proposition}

\theoremstyle{definition}
\newtheorem{definition}[theorem]{Definition}
\newtheorem{remark}[theorem]{Remark}

\newcommand{\ff}{\textbf{f}}
\newcommand{\hh}{\textbf{h}}
\renewcommand{\aa}{\textbf{a}}

\newcommand{\xx}{\textbf{x}}
\newcommand{\oo}{\textbf{0}}
\newcommand{\yy}{\textbf{y}}
\newcommand{\bb}{\textbf{b}}
\newcommand{\cc}{\textbf{c}}
\newcommand{\bt}{\begin{theorem}}
\newcommand{\et}{\end{theorem}}
\newcommand{\bc}{\begin{corollary}}
\newcommand{\ec}{\end{corollary}}
\newcommand{\bl}{\begin{lemma}}
\newcommand{\el}{\end{lemma}}
\newcommand{\br}{\begin{remark}}
\newcommand{\er}{\end{remark}}
\newcommand{\bd}{\begin{definition}}
\newcommand{\ed}{\end{definition}}
\newcommand{\be}{\begin{equation}}
\newcommand{\ee}{\end{equation}}
\newcommand{\ba}{\begin{array}}
\newcommand{\ea}{\end{array}}
\newcommand{\ra}{\rightarrow}

\newcommand{\R}{\mathbb{R}}
\newcommand{\RN}{\mathbb{R}^N}
\newcommand{\Z}{\mathbb{Z}}
\newcommand{\N}{\mathbb{N}}
\newcommand{\la}{\lambda}
\newcommand{\Id}{\mathrm{Id\,}}
\newcommand{\Ker}{\mathrm{Ker\, }}
\newcommand{\Coker}{\mathrm{Coker\,}}
\newcommand{\ind}{\operatorname{ind}}

\newcommand{\hof}{\hbox{\,of\, }}
\newcommand{\Ind}{\mathrm{Ind\, }}
\newcommand{\rk}{\operatorname{rk}}
\newcommand{\sk}{\medskip}
\newcommand{\hfor}{\hbox{ \,for\, }}
\newcommand{\hif}{\hbox{ \,if\, }}

\newcommand{\hand}{\hbox{ \,and\, }}
\newcommand{\im}{\operatorname{Im}}
\newcommand{\sign}{\operatorname{sign}}
\title{Topology and homoclinic trajectories of discrete dynamical
systems}

\author[Jacobo Pejsachowicz and Robert Skiba]{}

\subjclass{Primary: 39A28; 34C23, 58E07; Secondary: 37G20 47A53.}
 \keywords{Discrete dynamical systems, Homoclinics, Bifurcation, Index bundle, Fredholm maps}

 \email{jacobo.pejsachowicz@polito.it}
 \email{robo@mat.umk.pl}
\dedicatory{Dedicated to Petr P. Zabrejko}
\thanks{The first author is supported by  MIUR-PRIN 2009-Metodi variazionali e topologici nei fenomeni nonlineari. The second author is supported in part by Polish scientific grant N N201 395137}

\begin{document}

\begin{abstract}
We show that nontrivial homoclinic trajectories of
a family of discrete, nonautonomous,   asymptotically hyperbolic systems parametrized by a circle  bifurcate from a stationary solution if  the asymptotic stable bundles  $E^s(+\infty)$ and $E^s(-\infty)$ of the linearization at the stationary branch   are  twisted in different ways.
\end{abstract}

\maketitle

\centerline{\scshape Jacobo Pejsachowicz}
\medskip
{\footnotesize
 \centerline{Dipartamento di Matematica}
   \centerline{Politecnico di Torino}
   \centerline{Corso Duca Degli Abruzzi 24}
   \centerline{10129 Torino, Italy}
} 

\medskip

\centerline{\scshape Robert Skiba}
\medskip
{\footnotesize
 \centerline{Faculty of Mathematics and Computer Science }
   \centerline{Nicolaus Copernicus University}
   \centerline{Chopina 12/18, 87-100 Toru\'n, Poland}
}

\bigskip


\section{Introduction}

In this paper we will investigate  the birth of homoclinic
trajectories of  discrete nonautonomous dynamical systems  from
the  point of view of topological  bifurcation theory. This means
that instead of proving the existence of a homoclinic trajectory
for a single  dynamical system we will consider a one parameter
family  of discrete nonautonomous dynamical  systems  on $\RN$ having  $x\equiv 0$ as a
stationary trajectory and   show that, under appropriate conditions,
the dynamical systems with parameter values  close to a given point, must have  trajectories  homoclinic  to $0.$ Values of the parameter for which this occurs are called bifurcation points.

Bifurcation theory for various types of bounded solutions of
discrete  nonauton\-omous dynamical systems have been studied in
\cite{Ham,Ras} and more  recently in \cite{ Potz, Potz-1}. However
our approach is  different.  We will not look for homoclinics
bifurcating at a value of the parameter given apriori but instead we will
discuss the appearance of homoclinic solutions forced  by the
asymptotic behavior of the family of linearized equations at $0.$

When the family is asymptotically hyperbolic, the  asymptotic stable and  unstable subspaces  of  the linearized equations form vector bundles over the parameter space   which might  be nontrivial  when  the parameter space carries some nontrivial topology. We  will show that homoclinic trajectories  bifurcate from  a stationary solution  if  the asymptotic stable bundles  $E^s(+\infty)$ and $E^s(-\infty)$ of the linearization along the stationary branch   are  "twisted" in different ways.

Our results  require methods going beyond   the
classical Lyapunov-Schmidt reduction and spectral analysis at a
potential bifurcation point. While similar results can be proved  for general parameter spaces using  more sophisticated technology  from algebraic topology, here we will concentrate  to on
  the simplest topologically nontrivial parameter space, the circle. This is equivalent to considering families  of dynamical systems parametrized by an interval $[a,b]$ with the assumption that the systems  at $a$ and $b$ are the same.

Roughly speaking, we will first translate the problem of bifurcation of homoclinic trajectories
into a problem
of bifurcation from a trivial branch of zeroes for  a parametrized family of $C^1$-Fredholm maps.
Then  we will  consider the  {\em index bundle} of  the family of linearizations at points of the
trivial branch given by the stationary solutions of the equation. The index bundle of a family
of Fredholm operators is a refinement of the ordinary index of a Fredholm operator which takes into
account the topology of the parameter space.  A  special "homotopy variance" property  of the topological
degree for  $C^1$-Fredholm maps,  constructed in  \cite{Pej-Rab}  relates the nonorientability  of the index
bundle to bifurcation of zeroes. On the other hand, an elementary  index theorem, Theorem \ref{prop:ind}, allows
us to compute the index bundle in terms of  the asymptotic stable bundles  of the linearized problem, relating
in this way the appearance of homoclinics to the asymptotic behavior of coefficients of the linearized equations.
The precise result is stated in Theorem \ref{theorem} of Section $2$. An analogous approach applied to nonautonomous differential equations can be found in \cite{Pejs-2,Pejs-3}.

The paper is organized as follows. In the next section we
introduce  the problem and state our main result. In Section $3$
we recall  the concept of the index bundle and discuss  its
orientability.  In the fourth section, we
 compute the index
bundle of the family of operators associated to a family of
asymptotically hyperbolic nonautonomous dynamical systems.   In
Section $5$, we discuss  the {\it parity} of a path of Fredholm
operators of index $0,$ and we recall the construction in
\cite{Pej-Rab} of a topological degree theory for $C^1$-Fredholm
maps of index $0$ extending  to proper Fredholm maps the well
known Leray-Schauder degree. In Section $6,$  using  the
computation of the index bundle and the homotopy property  of the
topological degree constructed in \cite{Pej-Rab} we prove Theorem
\ref{theorem}. In  the seventh section, we illustrate
 Theorem \ref{theorem}
with  a non-trivial example.  Section $8$ is
devoted to comments and possible extensions of our results. 

\section{ The main result}

A nonautonomous discrete  dynamical system on $\R^N$
is defined by a doubly infinite
 sequence of  maps $\ff=\{f_n\colon \R^N\ra\R^N \mid
n\in \Z\}$. A trajectory of the system $\ff\colon \Z\times
\R^N\ra\R^N$ is a sequence $\xx=(x_n)$ such that
\begin{equation}\label{1} x_{n+1}=f_n(x_n).
\end{equation}

In  the
 terminology of \cite{Potz} \eqref{1} is a nonautonomous
difference equation whose solutions are trajectories of the
corresponding  dynamical system.

In what follows we will  always assume  that  the
$f_n$ are $C^1$ and that  $f_n(0) =0.$ Under this assumption the system has a stationary trajectory  $\xx=\oo, $ where $\oo$     is a sequence of zeroes. A  trajectory  $\xx=(x_n)$ of  $\ff$ is called {\it homoclinic} to $\oo$, or
simply a homoclinic trajectory, if   $\lim\limits_{n\ra \pm\infty}x_n=0.$ Under our assumptions the system $\ff$  has always a trivial  homoclinic trajectory. Namely, the stationary trajectory  $\oo.$  We will look for nontrivial homoclinic trajectories.

A   natural function space for the study of homoclinic trajectories is the
Banach space
\begin{equation*}
{\bf c}(\R^N):=\{\xx\colon \Z\ra \R^N \mid \lim_{|n|\ra \infty }
x_n =0\}
\end{equation*}
equipped with the norm $\|\xx\|\!\!:=\sup_{k\in \Z}|x_n|$.
Any   homoclinic trajectory of $\ff$  is naturally an element of  this space.
 Moreover, each  dynamical  system $\ff$  induces  a nonlinear Nemytskii (substitution) operator
\be\label{2} F\colon \cc(\RN)\ra\cc(\RN)\ee
 defined by $F(\xx)=(f_n(x_n)).$ Under some natural  assumptions (see below)  $F$ becomes
$C^1$-map
 such that $F(\oo)= \oo.$ In this way  nontrivial homoclinic trajectories become
 the nontrivial solutions of the equation $S\xx -F(\xx) =\oo,$ where $$ S\colon \cc(\RN)\ra\cc(\RN)$$
 is the shift operator $S(\xx) = (x_{n+1}).$

The linearization of the system $\ff$ at the stationary solution
$\oo$ is the nonautonomous linear dynamical system
$\aa\colon \Z\times \R^N\ra\R^N$ defined by the sequence of
matrices $(a_{n})\in \R^{N\times N},$ with $a_n = Df_n(0)$. The
corresponding linear difference equation is
\begin{equation}\label{1'}x_{n+1} = a_{n}x_n.
\end{equation}

If  $f_n=f,$ for all $n\in \Z,$
the system is called autonomous.
We will deal only with discrete nonautonomous dynamical systems
whose linearization at $\oo$ is asymptotic  for $n \ra \pm \infty$
to an autonomous  linear dynamical system associated to a
hyperbolic matrix.  We will call systems with this property {\it
asymptotically hyperbolic}.

Let us recall that an invertible matrix   $a$ is
called {\it hyperbolic} if $a$ has no eigenvalues of norm one,
i.e., $\sigma(a)\cap \{|z|=1\}=\emptyset.$ The spectrum
$\sigma(a)$ of an hyperbolic matrix $a$ consists of  two
disjoint closed subsets $\sigma(a)\cap\{|z| < 1\}$ and
$\sigma(a)\cap\{|z|> 1\}$, so $\R^N$ has the $a$-invariant
spectral decomposition $\R^N=E^s(a) \oplus E^u(a)$, where $E^s(a)
$ (respectively $E^u(a)$)  is the real part of   sum of the generalized
eigenspaces corresponding to the part of the spectrum of $a$ inside
the unit disk (respectively outside of the unit disk). It is easy
to see that $ \zeta \in E^s(a)$ if and only if
$\lim\limits_{n\ra\infty}a^n\zeta=0$.
The unstable subspace
$E^u(a) $  has a similar characterization, i.e., $\zeta\in E^u(a)$
if and only if $\lim\limits_{n\ra\infty}a^{-n}\zeta=0.$

When the linearized system is asymptotically  hyperbolic, the map
$G=S-F$  becomes a Fredholm map  (at least in a neighborhood of $\oo$) which will allow us to apply the results of
general bifurcation theory for Fredholm maps to our problem by relating the corresponding bifurcation invariants
to the asymptotic behavior of the linearization at  $\pm \infty.$

Let us describe   precisely our setting and assumptions.

A  {\it continuous  family of  $C^1$-dynamical systems parametrized by  the unit circle $S^1$} is  a sequence of  maps
\begin{equation}\label{cfd}
\ff=\{f_n\colon S^1\times \R^N \ra \R^N \mid n\in\Z\}
\end{equation}
such that $f_n$ is  differentiable with respect to the second
variable and,  for all $n\in\Z, \ 0\leq j\leq 1,$
the map $(\la,x)\mapsto \displaystyle \frac{\partial^jf_n}{\partial x^j}(\la,x)$ is continuous.

To put it shortly, a continuous family of  $C^1$-dynamical systems  is a continuous map
$$\ff\colon  \Z\times  S^1\times \R^N\ra \R^N,$$
differentiable in the third variable and such that all the
partials  depend continuously on $(\la,x).$ We will use  $\ff_\la$
to denote the dynamical system corresponding to the  parameter
value $\la.$

\begin{remark} {\rm Alternatively one can think of $\ff$ as a double infinite sequence of maps $f_n \colon [a,b]\times \R^N \ra \R^N,$  such that $f_n(a,x)=f_n(b,x)$ for all $n\in \Z.$ } \end{remark}

 Pairs $(\la,\xx)$  which solve the
parameter-dependent difference equation:
\begin{equation}\label{main-system}
x_{n+1}=f_n(\la,x_n),\;\text{ for all }n\in\Z,
\end{equation}
will be  called  {\it homoclinic solutions}. Equivalently,  $(\la, \xx)$  is a homoclinic solution of \eqref{main-system} if  $\xx=(x_n)$ is a homoclinic trajectory of the dynamical system $\ff_\la.$

Homoclinic solutions of
(\ref{main-system}) of the form $(\la,{\bf 0})$ are called
trivial and the set $S^1\times\{\bf 0\}$ is called the {\it trivial or stationary branch.}
We are interested in nontrivial homoclinic solutions.

We will assume that the family
 $\ff\colon  \Z\times  S^1\times \R^N\ra \R^N$ of dynamical systems satisfies the following conditions:
\begin{enumerate}

\item[(A0)] For all $\lambda\in S^1$ and $n\in \mathbb{Z}$,
$f_{n}(\la,0)=0$.

\item[(A1)] For any $M>0$ and $\varepsilon>0$ there exists a $\delta>0 $
such that
for  all $(\la,x), (\mu,y) \in S^1\times \bar{B}(0,M)$
(\footnote{Given a normed space $\mathbb{E}$, $\bar{B}(x,r)$ and
$B(x,r)$, where $x\in \mathbb{E}$ and $r>0$, denote the closed and
open ball around $x$ of radius $r$ in $\mathbb{E}$,
respectively.}) with $ d\big((\la,x),(\mu,y)\big)<\delta$ and all
$j, \,0\leq j \leq 1,$

\begin{equation*}
\sup_{n\in \Z} \Bigg\|\frac{
\partial^j f_n}  {\partial x^j}(\la ,x) - \frac{\partial^j f_n} {\partial x^j}  (\mu ,y)\Bigg\|<\varepsilon.\end{equation*}

Here $d$ is the product distance in the metric space $S^1\times \R^N.$

\item[(A2)] For all bounded $\Omega\subset S^1\times \R^N$ one has
\begin{equation*}
\sup_{(n,\la,x)\in \Z\times \Omega} \Bigg\|\frac{\partial f_n}
{\partial x } (\la ,x) \Bigg\|<\infty.
\end{equation*}
\item[(A3)]  Let
$a_n(\la):=\displaystyle\frac{\partial f_n}{\partial x} (\la,0).$  As  $n\ra
\pm\infty$ the family of matrices  $a_n(\lambda)$ converges uniformly to a family of hyperbolic matrices
 $a(\lambda,\pm\infty).$ Moreover, for some,  and hence for all $\la \in S^1,$ $a(\la, +\infty)$ and $a(\la,-\infty)$
 have the same number of eigenvalues (counting  algebraic multiplicities) inside of the unit disk.
\item[(A4)]
There exists $\lambda_0\in S^1$ such that
\begin{equation}\label{3}
x_{n+1} = a_{n}(\la_0) x_n
\end{equation}
admits only the trivial solution $(x_n\equiv
0)_{n\in \mathbb{Z}} .$\end{enumerate}

By $(A3)$ the map $\la \ra
a(\la,\pm\infty)$ is a continuous family of hyperbolic matrices.
Since there are no eigenvalues  of
 $a(\la,\pm\infty)$ on the unit circle, the projectors to the spectral subspaces corresponding to the spectrum inside and  outside  the unit disk depend continuously on the parameter $\la$ (see \cite{Kat}). It is well known that  the  images of a continuous family of  projectors   form   a vector bundle over the parameter space \cite{La}.   Therefore,  the  vector spaces $E^s(\la,\pm\infty)$ and $E^u(\la,\pm\infty)$ whose elements are the generalized real
eigenvectors of $a(\la,\pm\infty)$ corresponding to the
eigenvalues with absolute value smaller (respectively greater)
than $1$ are fibers of a pair of vector bundles $E^s(\pm\infty)$
and $E^u(\pm\infty)$ over $S^1$ which decompose the trivial bundle
$\Theta(\R^N)$ with fiber $\R^N$ into a direct sum:
\begin{equation}\label{dirsum}
E^s (\pm \infty)\oplus  E^u(\pm \infty) = \Theta(\R^N).
\end{equation}

In what follows  $E^s (\pm \infty)$  and $E^u(\pm \infty)$  will
be called  \textit{stable and unstable}  asymptotic bundles at
$\pm\infty.$

Our main theorem relates the appearance of homoclinic solutions to
the topology of the asymptotic stable bundles $E^s(\pm \infty)$.
Due  to
 relation \eqref{dirsum} the consideration of the
unstable bundles would give the same result.

In what follows, for notational reasons,  it will be convenient for
us to work  with the multiplicative group $\Z_2=\{1,-1\}$ instead of the standard additive $\Z_2=\{0,1\}.$

A vector bundle over $S^1$ is orientable if and only if it is trivial, i.e., isomorphic to a product $S^1\times\R^k.$
 Moreover, whether a given vector bundle
 $E$ over $S^1$ is trivial or is not is determined by a topological invariant $w_1(E)\in\Z_2.$

In order to define $w_1(E)$  let us identify $S^1$ with the quotient of
an interval $ I= [a,b]$ by its boundary $\partial I = \{a,b\}.$ If
$p\colon [a,b]\ra S^1= I/\partial I$  is the projection,  the
pullback bundle $p^*E =E'$ is the vector  bundle over $I$ with
fibers $ E'_t =E_{p(t)}.$  Since $I$ is contractible to a point,
$E'$ is trivial and the choice of an isomorphism between $E'$ and
the product bundle provides $E'$ with a frame, i.e., a basis
$\{e_1(t),...,e_k(t)\}$ of $E'_t $ continuously depending on $t$.
Since  $E'_a =E_{p(a)}=E_{p(b)}=E'_b,$  $\{e_i(a)\mid 1\leq i\leq k\}$
 and $\{e_i(b)\mid 1\leq i\leq k\}$ are two bases  of the same vector space.
We define $w_1(E)\in\Z_2$ by
\be \label{whitney} w_1(E):=\sign
\det C,\ee where $C$ is the matrix expressing  the basis $\{e_i(b)\mid 1\leq
i\leq k\}$ in terms of  the basis $\{e_i(a)\mid 1\leq i\leq k\}$.

It is easy to see that $w_1(E)$ is independent from the choice of
the frame. We claim that  $w_1(E)=1$ if and only if $E$ is
trivial. The if part is an immediate consequence of the definition
of $w_1(E).$  On the other hand,  if $w_1(E)=1,$ then $\det C>0$
and there exists a path $C(t)$ with $C(a)=C$ and
$C(b)=\mathrm{Id}$. Now, $f_i(t) = C(t) e_i(t)$ is a  frame such
that $f_i(a)= f_i(b)$ and hence $\Phi(t, x_1,\ldots,x_k) =
\left(t,\sum x_i f_i(t)\right)$ induces  an isomorphism over $S^1$  between the
product bundle $S^1\times \R^k$ and $E.$ Thus $E$ is trivial.

\begin{remark} Under  the isomorphism $H^1(S^1;\Z_2)\cong\Z_2$, $w_1(E)$
can be identified with the first Stiefel-Whitney class of $E$.
\end{remark}

Our main result is:
\begin{theorem}\label{theorem} If the system \eqref{main-system}
verifies $(A0)$--$(A4)$ and if \be w_1(E^s(+\infty))\neq
w_1(E^s(-\infty)), \ee then for all $\varepsilon$ small enough
there is a homoclinic solution $(\lambda,\emph{\xx})$ of
\eqref{main-system} with $\|\emph{\xx}\|=\varepsilon$.
\end{theorem}

The proof will be presented in Section $6$.
\vskip 10pt

A point $\la_*\in S^1$ is a {\it bifurcation point} for homoclinic
solutions of \eqref{main-system}  from the    stationary branch
$(\la,\oo)$ if in every neighborhood of $(\la_*,\oo)$  there is a
point of a nontrivial homoclinic solution $(\la,\xx)$
 of $x_{n+1}=f_n(\la,x_n).$

By  Theorem \ref{theorem} we can find a sequence of  nontrivial homoclinic solutions $(\la_n,\xx_n)$ of
\eqref{main-system}  such that $\|\xx_n\|\ra 0.$ Since $S^1$ is compact $\la_n$ possesses a subsequence
converging to some $\la_*\in S^1.$ Hence  we obtain:

\begin{corollary}\label{cor1}
Under the assumptions of Theorem \emph{\ref{theorem}} there exists
at least one bifurcation point $\la_*\in S^1$ of nontrivial
homoclinic solutions from the  branch of stationary solutions. In
other words there exists a $\la_*\in S^1$  and a sequence $(\la_k,
\emph{\xx}_k)$ such that $\la_k\ra\la_*$ and $\emph{\xx}_k\neq
\emph{\oo}$ is a nontrivial   homoclinic trajectory of
$\emph{\ff}.$
\end{corollary}

Let us observe that if $\ff\colon  \Z\times  S^1\times \R^N\ra
\R^N$  verifies the assumptions of Theorem \ref{theorem}  and
$\tilde\ff \colon  \Z\times S^1\times \R^N\ra \R^N $ is defined by
$\tilde\ff = \ff + \hh, $ where $$\hh =(h_n) \colon  \Z\times
S^1\times \R^N\ra \R^N,$$   verifies  $(A0)$--$(A2)$ and  moreover

\begin{itemize}
\item[$(A3')$]  $ \displaystyle \frac{\partial h_n} {\partial x} (\la,0) \ra 0$ as $n\ra \pm\infty$  uniformly on
$\la,$
\item[$(A4')$]  $\displaystyle \sup\limits_{n\in \Z} \Bigg\|\frac{\partial h_n} {\partial x} (\la_0,0)\Bigg\|$ is small  enough,
\end{itemize}
then also $\tilde \ff$ verifies Assumptions $(A0)$--$(A4).$
Indeed, $(A3')$ for  $\hh$ implies   $(A3)$  for $\tilde\ff .$ On
the other hand, it is shown in the proof of Theorem \ref{theorem}
that $(A3)$ and $(A4)$  together imply  that  the operator
$L_{\la_0} \colon \cc(\RN) \ra \cc(\RN)$ defined by
$$L_{\la_0}\xx=(x_{n+1}-a_n(\la_0)x_n) $$
is  invertible.  Now, that $\tilde \ff$ verifies  $(A4)$ follows
from $(A4')$ and the fact that the set of all invertible operators
is open.

Summing up we have:

\begin{corollary}\label{cor2}
If \emph{$\ff$}
verifies  the assumptions of Theorem \emph{\ref{theorem}}  then any
perturbation \emph{$\tilde\ff=\ff+\hh$}
 as above must have  nontrivial homoclinic solutions  bifurcating  from the stationary  branch  at some point of the parameter space.
\end{corollary}

\section{The index bundle}
Let us recall that  a bounded operator $T\in \mathcal{L}(X,Y)$
(\footnote{By $\mathcal{L}(X,Y)$ we will denote the space of
bounded linear operators between two Banach spaces X and Y.}) is
Fredholm if it has finite dimensional kernel and cokernel.  The
index of a Fredholm operator  is by definition $\ind T:=\dim \Ker T -
\dim \Coker T.$ The Fredholm operators will be denoted by
$\Phi(X,Y)$ and those of index $0$ by $\Phi_0(X,Y).$

The index bundle generalizes to the case of families of Fredholm
operators the concept  of index of a single Fredholm operator. If
a family $L_\la$ of Fredholm operators depends continuously on a
parameter $\la$ belonging to some topological space $\Lambda$  and
if the kernels $\Ker L_\la$ and cokernels $\Coker L_\la$  form two
vector bundles $\Ker L$ and $\Coker L$ over $\Lambda,$ then,
roughly speaking, the index bundle is  $\Ker L-\Coker L$  where
one has to give a meaning
to the  difference by working in an
appropriate group generalizing $\Z$. We will first define
such a  group and then will see how to handle the case where
the kernels do not form a vector bundle.

If $ \Lambda$ is a compact topological space, the Grothendieck
group $KO(\Lambda )$ is the group completion of the abelian
semigroup $\text{Vect} (\Lambda )$ of all isomorphisms classes of
real vector bundles over  $\Lambda.$ In other words, $KO(\Lambda
)$   is the quotient of the semigroup $ \text{Vect}(\Lambda)$
$\times \text{Vect}(\Lambda)$ by the diagonal sub-semigroup.  The
elements of $KO(\Lambda)$  are called  virtual bundles.  Each
virtual bundle can be written  as a difference $[E] - [F]$ where
$E, F$ are vector bundles over  $\Lambda  $ and $[E]$ denotes the
equivalence class of $(E,0).$  Moreover, one can show that $ [E] -
[F]= 0$ in $KO(\Lambda )$ if and only if the two vector bundles
become isomorphic after the addition of a  trivial vector bundle
to both sides. Taking complex vector bundles instead of the real
ones leads to  the complex Grothendieck  group denoted by
$K(\Lambda).$ In what follows the trivial bundle  with fiber
$\Lambda\times V$ will be denoted by  $\Theta(V).$ The trivial bundle,   $\Theta(\R^N),$ will be simplified to  $\Theta^N.$

Let $X,\ Y$ be real Banach spaces and let  $L\colon
\Lambda\rightarrow \Phi(X,Y)$  be a continuous family of Fredholm
operators. As before $L_{\la}\in \Phi(X,Y)$ will  denote the value
of $L$ at the point ${\la} \in  \Lambda$. Since $\Coker L_\la$
is
finite dimensional, using  compactness of $ \Lambda,$ one can find
a finite dimensional subspace $V \hof\,  Y$ such  that
\begin{equation} \label{1.1}
\hbox{\rm Im}\,L_\la+ V=Y  \ \hbox{\rm for all }\  \la \in
\Lambda.
\end{equation}

Because of the transversality condition \eqref{1.1}  the family of
finite dimensional subspaces $E_{\la}=L_{\la}^{-1}(V)$ defines a
vector bundle over $ \Lambda$ with  total space
\[E= \bigcup_{\la \in \Lambda}\, \{\la\} \times E_\la.\]
Indeed, the  kernels of a family of surjective Fredholm operators
form a finite dimensional vector bundle \cite{La}.  Denoting with
$\pi $  the canonical projection of $Y$ onto $Y/V,$  from
\eqref{1.1} it follows that the operators $ \pi {L}_{\la}$  are
surjective with $\Ker\pi {L}_\la=E_\la,$
which  shows that $E\in
Vect(\Lambda).$

We define the {\it index bundle}  $\Ind L$ by:
\begin{equation} \label{defind}
\Ind L= [E]-[\Theta (V)] \in KO(\Lambda).
\end{equation}

Notice that the index bundle of a family of Fredholm operators of
index $0$ belongs to the reduced Grothendieck group $
\widetilde{KO}(\Lambda )$ which,
by definition,  is the  kernel of the rank
homomorphism $rk \colon KO(\Lambda) \rightarrow \Z$ given by
\begin{equation*}
\rk([E] -[F]) = \dim E_\la -\dim F_\la.
\end{equation*}
We will mainly, but not always, work with families of Fredholm
operators of index $0$. If $\Lambda =pt$ consists of just one
point, then the rank homomorphism $\rk$ is an isomorphism and the
index bundle coincides with the ordinary numerical index $\ind L =
\dim \Ker L - \dim \Coker L $ of a Fredholm operator $L$. The
index bundle enjoys the same nice properties of the ordinary
index. Namely, homotopy invariance, additivity with respect to
directs sums, logarithmic property under composition of operators.
Clearly it vanishes if $L$ is a family of isomorphisms. We will
use these properties in the sequel. The precise statements and
proofs can be found in \cite[Appendix A]{Pejs-1}.

It can be shown  that  any element $\eta \in
\widetilde{KO}(\Lambda)$ can be written as $[E] -[\Theta^N ].$ Moreover,
$[E] -[\Theta^N ] = [E'] -[\Theta^M] $ in $\widetilde{KO}(\Lambda
)$ if and only if there exist two trivial bundles $\Theta $ and
$\Theta'$ such that $E\oplus \Theta$ is isomorphic to  $E'\oplus
\Theta',$ (see \cite[Theorem 3.8]{Hus}).

The obstruction  $w_1(E)$  to the triviality of  vector bundle $E$
over $S^1$ defined in Section $2$  induces  a well defined
homomorphism $w_1\colon \widetilde{KO}(S^1) \ra \Z_2$ by putting
\begin{equation}\label{morewhit}
w_1([E] -[F]) = w_1(E) w_1(F).
\end{equation}

Indeed, taking $\Lambda =S^1$ we observe that  $w_1(E)$ remains
unmodified under addition of a trivial  vector bundle which, on
the basis of the above discussion,  proves that \eqref{morewhit}
is well defined.

\begin{proposition}\label{isomorphism-w}
The homomorphism   $w_1\colon \widetilde{KO}(S^1) \ra \Z_2$ is an
isomorphism.
\end{proposition}
\begin{proof} This  follows again  from the above discussion and the fact
that $w_1(E)=1$ implies that $E$ is a trivial vector bundle over
$S^1.$
\end{proof}

\section{The index  bundle of the family of operators associated to  linear asymptotically hyperbolic  systems}

In this section we will deal only with  linear asymptotically
hyperbolic systems $\aa \colon  \Z  \times S^1\ra GL(N)$, where
$GL(N)$ is the set of all invertible matrices in $\R^{N\times N}.$
This  means:
\begin{itemize}
\item[(a)] As $n \ra\pm \infty$ the sequence  $\aa(\la)=(a_n(\la))$ converges
uniformly with respect to $\la\in S^1$ to a family of matrices
$a(\la,\pm \infty)$.
\item [(b)]  $a(\la,\pm \infty)\in GL(N)$ is hyperbolic for all $\la \in S^1.$
\end{itemize}

Given a family $\aa$ of asymptotically hyperbolic systems parametrized by $S^1$ let us  consider the family of linear operators
$$L = \{L_\la \colon \cc(\RN)\ra \cc(\RN) \mid \la \in S^1 \} $$     defined by $ L_\la =S-A_\la,$ where $S$ is the shift operator and
$$A_\la \colon {\bf c}(\R^N)\ra {\bf c}(\R^N)$$ is defined by
$A_\la \xx\!:=(a_{n}(\la)x_n).$

 Since the  sequence  $(a_n(\la))$ converges uniformly,  it is bounded,   from which  follows immediately that   $A_\la$ and  $L_\la $  are
 well defined bounded operators. Moreover
 it is easy to
see that the map $A\colon S^1\ra \mathcal{L}({\bf c}(\R^N),{\bf
c}(\R^N))$ defined by $A(\la):=A_{\la}$ is continuous with respect
to the norm topology of $\mathcal{L}({\bf c}(\R^N),
{\bf c}(\R^N))$. Hence the same holds for the family $L.$

Clearly, $\xx=(x_n)\in \cc(\R^N)$ verifies a linear
difference equation $x_{n+1}=a_n(\la)x_n$ if and only if $L_{\la}\xx=0$.

By the discussion in the previous section  the families  $a(\la,\pm \infty)\in GL(N)$
define two vector bundles $E^s (\pm \infty)$ over $S^1.$  The next theorem  relates the index bundle of the family $L$ to  $E^s (\pm \infty).$

\begin{theorem}\label{prop:ind} Let
$ \aa\colon  \Z \times S^1 \ra GL(N)$ be
a continuous  map verifying
$(a)$ and $(b).$ Then the family $L\colon S^1\ra \mathcal{L}({\bf
c}(\R^N),{\bf c}(\R^N))$ verifies:

\begin{itemize}
\item [$(i)$] $L_\la$ is a Fredholm operator  for
all $\lambda\in S^1$.
\item [$(ii)$] $\Ind L= [E^s (+ \infty)]-[E^s(-\infty)] \in
{KO}(S^1)$.
\end{itemize}
\end{theorem}

\begin{remark}
In the proof of Theorem \ref{prop:ind} we will  also compute the
index  of  $L_{\la}$ in terms of dimensions of the stable spaces
at $\pm \infty.$  This is far from being new, and similar
computations using exponential dichotomies can be found in many
places, e.g., \cite{Ba,Sa}.  Here we are not interested  in the
index but rather in the index bundle and our theorem can be
considered an extension to the case of families of the
computations  quoted above.
\end{remark}

\begin{proof}
 Let $\bar{\aa} \colon S^1\times \Z \ra GL(N) $
be defined by
\begin{equation}\label{const}
 \bar{\aa}(\la,n)=( \bar a_n(\la))=
\begin{cases} a(\la, +\infty) & \hif\; n\geq 0,
\\ a(\la,- \infty)
&\hif\; n<0.
\end{cases}
\end{equation}

Put $X:={\bf c}(\R^N).$ Fix $\la\in S^1$ and  denote  by
   $\bar A_\la \in \mathcal{L}(X,X)$
the  operator associated to $\bar{\aa}_\la.$
We claim that the
operator  $K_\la=A_\la -\bar A_\la$ is a compact operator. To this
end, we will show that $K_\la$ is the limit  (in the norm topology
of $\mathcal L (X,X)$) of a  sequence of operators $\tilde
K^m_\la$ with finite dimensional range.
 We observe that $K_\la$ is defined by  $K_{\la}\xx=(k_n(\la) x_n),$  where
 $k_n(\la)=a_n(\la)-\bar a_n(\la)$ and define

\begin{equation}
\label{finite} \tilde K^m_\la\xx=
\begin{cases} k_n(\la) x_n & \hif \;|n|\leq m,  \\ 0 &\hif \;
|n|>m.
\end{cases}
\end{equation}
Clearly $\im \tilde K^m_\la$ is  finite dimensional. We are to prove that
\begin{equation}\label{wz} \sup_{\|\xx\|=1}
\|(K_{\la}-\tilde
K^m_{\la})\xx\|\xrightarrow[m\rightarrow\infty]{}0,
\end{equation}
for $\xx\in X$. Observe that
\begin{align}\label{wz1}
\begin{split}
\|(K_{\la}-\tilde K^m_{\la})\xx\|=\sup_{|n|>m}\|k_n(\la)x_n\| \geq
\sup_{|n|>m+1}\|k_n(\la)x_n\|=\|(K_{\la}-\tilde
K^{m+1}_{\la})\xx\|,
\end{split}
\end{align}
for all $m\in \N$. Since
\begin{equation*}
\lim_{|n|\ra \infty} k_n(\la)=0,
\end{equation*}
we infer that for all $ \varepsilon>0$ there exists $n_0>0$ such
that for all $|n|> n_0$ and $\|\xx\|=1$ one has
\begin{equation*}
\|k_n(\la)x_n\|<\varepsilon.
\end{equation*}
Consequently, for all $\varepsilon>0$ there exists $n_0>0$ such
that
\begin{equation}\label{wz2}
\sup_{
\|\xx\|=1}\|(K_{\la}-\tilde K^{n_0}_{\la})\xx\|\leq\varepsilon.
\end{equation}
Now taking into account (\ref{wz1}) and (\ref{wz2}), we deduce
that for all $\varepsilon>0$ there exists $n_0>0$ such that for
all $m\geq n_0$ one has
\begin{equation} \sup_{ \|\xx\|=1} ||(K_{\la}-\tilde
K^{m}_{\la})\xx||\leq\varepsilon,
\end{equation}
which proves (\ref{wz}) and the compactness of the operator $K_\la.$

Let $\bar L_\la =S-\bar{A}_\la.$  Then $L_\la -\bar L_\la= K_\la$  and hence  the family $L$ differs
from the family $\bar L$ by a family of compact operators. Therefore    $L_\la $ is Fredholm if and only if  $\bar L$
 is Fredholm  and moreover the  homotopy invariance of the index bundle applied to the homotopy $ H(\la,t) = \bar L_\la + t K_\la$
 shows that $\Ind \bar L = \Ind L.$  Hence in order to prove the theorem  we can assume  without loss of generality
 that $\aa$ has already the special  form of \eqref{const}, which we will do from now on.
   Let
\begin{align*}
\cc^+_k&=\{\xx \in {\bf c}(\R^N) \mid x_i=0 \hfor i<k\},\\
\cc^-_k&=\{\xx \in {\bf c}(\R^N) \mid x_i =0 \hfor i>k\}.
\end{align*}
Both $\cc^\pm_k$ are closed subspaces of ${\bf c}(\R^N)$. The space
$\cc^+_k$ can be isometrically identified with
$$ \cc_k(\R^N):=\{\xx\colon[k,\infty)\cap\Z  \ra \R^N \mid
\lim_{n\ra \infty } x_n =0\}$$
and similarly for $\cc^-_k.$

Put $ X^+ =Y^+ = \cc_0^+$ and $X^- = \cc^-_0, \ Y^- = \cc^-_{-1}$.
Let us consider four linear operators
$I\colon Y^- \oplus Y^+ \ra X,$
$J\colon X \ra X^- \oplus X^+,$
$L^+ _\lambda \colon X^+ \ra Y^+$ and $L^- _\lambda \colon X^- \ra
Y^-$ defined respectively by
\begin{align*}
I(\xx,\yy)&=\xx+\yy,
\\
J(\xx)(n)&=\begin{cases} (x_0,x_0)& \hif\; n=0,\\(x_n,0) & \hif
\;n<0,
\\
(0,x_n)&
\hif \;n>0,
\end{cases}
\\
(L_\la^+\xx)(n)&=
\begin{cases} x_{n+1} -a(\la,+ \infty)x_n
& \hfor n\geq 0,
\\ 0
&\hfor n<0,
\end{cases}
\\
(L^-_{\la}\xx)(n)&=
\begin{cases}
0 & \hfor n>-1,
\\  x_{n+1} - a(\la,- \infty)x_n
& \hfor n\leq -1.
\end{cases}
\end{align*}
We  decompose $L_{\la} \colon X\ra X$ via the following
commutative diagram:

\begin{equation}
\label{eq:homdiag}
\begin{diagram}
\node{ X^{-} \oplus X^{+}}  \arrow{e,l}{L^-_\lambda \oplus
L^+_\lambda} \node{Y^{-} \oplus Y^{+}} \arrow{s,b}{I}
\\
\node{X}\arrow{n,t}{J}\arrow{e,b}{L_\lambda} \node{X.}
\end{diagram}
\end{equation}
The commutativity of
 diagram \eqref{eq:homdiag} is easy to check.

Indeed, one has

\begin{align*}
I(L_{\la}^-\oplus L_{\la}^+)J\xx(n)=
L_{\la}^-J\xx(n)+L_{\la}^+J\xx(n)=\begin{cases}
(L_{\la}^+\xx)(n) & \hif\;n\geq 0,\\
(L_{\la}^-\xx)(n) & \hif\;n<0,\end{cases}
\end{align*}
which is the same as
\begin{equation} \label{operator}
(L_{\la} \xx)(n) =
\begin{cases}
x_{n+1} - a(\la, +\infty)x_n  & \hif\; n\geq 0,
\\x_{n+1}- a(\la,- \infty)x_n &\hif\;  n<0.
\end{cases}
\end{equation}

Next, we will show that $L_\la^\pm \colon X^\pm \ra Y^\pm$ are Fredholm
and we will compute the index bundles of $L^\pm $.

For $L_\la^+$ this is the content of the following  Lemma:

\begin{lemma}\label{ab-lemma} \emph{(\cite[Lemma
2.1]{Ab-Ma2})} Let  $a\in GL(N)$ be an  hyperbolic matrix.  Then
the  operator $S-A\colon \cc_0^+\ra \cc_0^+,$
defined by
\begin{equation*}
((S-A)\emph{\xx})(n)=
\begin{cases}
x_{n+1}-ax_{n} & \emph{\hif}\; n\geq 0, \\
0 &\emph{\hif}\; n<0,
\end{cases}
\end{equation*}
is surjective with
\begin{equation*}
\ker(S-A)=\{\emph{\xx}\in \cc_0^+\mid x_{n+1}=a^nx_0 \text{ for all
$n\geq 0$ and } x_0\in E^s(a)\}.\end{equation*}
\end{lemma}

This lemma was proved in \emph{\cite[Lemma
2.1]{Ab-Ma2}}  by  constructing an explicit right inverse to the operator
 $S-A\colon \cc_0^+\ra \cc_0^+.$
 \qed

By Lemma \ref{ab-lemma}
\begin{equation}\label{ker+}
\Ker L^+_{\la}=\{\xx\in X^+ \mid x_n =
a(\la,+\infty)^n x_0
 \hand x_0\in E^s(\la,+\infty)\}.
\end{equation}
Hence the transformation  $\xx \mapsto x_0 $ defines an
isomorphism between $\Ker L^+ $ and $E^s(\la,+\infty),$ which is
finite dimensional. Being $\Coker L_\la =0,$
$L^+_{\la}$ is
Fredholm with $\ind L^+_{\la}=\dim E^s(\la,+\infty).$ Clearly the
index bundle  $\Ind L^+ = [E^s(+\infty)].$

  We will reduce the calculation of $\Ind L^-$ to
Lemma \ref{ab-lemma} as follows:

Put $Y^-:= \cc^-_{-1} $
and $X^-:=\cc^-_0$
and consider the  family of isomorphisms \\ $ B
=\{B_{\la} \colon Y^- \ra Y^-\} $  defined by
\begin{equation*}
(B_{\la}\xx)(n)=
\begin{cases}
0
& \hfor n>-1,  \\
- a^{-1}(\la,-\infty)x_n
&\hfor n\leq -1.
\end{cases}
\end{equation*}
We compose $L^-_{\la} \colon X^- \ra Y^-$ on the right with the
isomorphism $B_{\la} \colon Y^- \ra Y^-$ followed by the negative
shift $ S^{-1}$ viewed as an operator from $Y^-$ to $X^-$.  Since
both operators are isomorphisms the composition does not affect
 the Fredholm property. On the other hand
considering $S^{-1}$ as a constant  family of
isomorphisms, by logarithmic property of the index bundle, $\Ind
S^{-1}B L^- =\Ind L^-.$
Hence the index bundle of $L^-$ coincides with the index bundle of the family $D=S^{-1}B L^-$.
Observe now that, if $\xx\in Y^-$, then
\begin{equation*}
(B_{\la} L^-_{\la})\xx)(n)=
\begin{cases}
0 & \hfor n>-1,  \\
x_n  - a^{-1}(\la,-\infty)x_{n+1} &\hfor n\leq -1.
\end{cases}
\end{equation*}
But since $S^{-1}\xx=(x_{n-1}),$ one obtains
\begin{equation*}
(D_{\la}\xx)(n)=
\begin{cases}
0 & \hfor n>0,  \\
x_{n-1} -  a^{-1}(\la,-\infty)x_n &\hfor n\leq 0.\end{cases}
\end{equation*}

Thus  $D_{\la} \colon X^-\ra X^-$  is
the same type of operator as $L^+_{\la}$ but with $n$ going  from $0$ to
$-\infty.$
By Lemma \ref{ab-lemma}, each  $D_\la$ is surjective.  Moreover,  \begin{equation}\label{ker-} \Ker D_{\la} =\Ker L^-_{\la} =\{\xx\in X^- \mid
 x_n=a(\la,-\infty)^{n}x_0
 \hand x_0\in E^u(\la,-\infty))\}\end{equation} is isomorphic to
$ E^u(\la,-\infty)$.

  Summing up, we have obtained that
$\Ind\,L^+= [E^s( +\infty)]$  and  $\Ind L^-= [ E^u( -\infty)].$
In particular we have
\begin{equation}\label{dim+}
\ind L^+_\la = \dim  E^s(\la,+\infty)
\quad \text{and}\quad \ind
L^-_\la = \dim  E^u(\la,-\infty).
\end{equation}

With this at hand we  can compute the index bundle of $L$  completing  the proof of the theorem. Let us notice firstly that $I$ and $J$ are Fredholm operators.
Indeed,  $ I \colon   Y^-\oplus Y^+ \ra  X $   is clearly an
isomorphism,  and  the map $ J\colon X \ra X^-\oplus
X^+$ is a  monomorphism  whose image is given by $\im J
=\{(\aa,\bb)\in X^-\oplus X^+\mid a_0 \nolinebreak=b_0\}.$ Putting
$P\colon X^-\oplus X^+\ra \R^N$ by $P(\aa,\bb):=a_0-b_0,$ for
$\aa\in X^-$ and $\bb\in X^+$, one obtains that $\im J=\Ker P.$
But since $P$ is an epimorphism, we deduce that $\Coker
J=X^-\oplus X^+/\Ker P \simeq \R^N$ and therefore $J$ is Fredholm
of index $-N$. From the commutativity of
 diagram \eqref{eq:homdiag} and \eqref{dim+} it
follows that $L_{\la}=I(L^-_{\la}\oplus L^+_{\la})J$ is
Fredholm and
\begin{equation}\label{dim}
\begin{array}{ll}
&\ind(L_{\la})=\ind(I)+\ind(L^-_{\la}\oplus L^+_{\la})+\ind(J)=\\
&\dim  E^s(\la,+\infty)+\dim  E^u(\la,-\infty)-N=\\ &\dim
E^s(\la,+\infty)-\dim  E^s(\la,-\infty).
\end{array}
\end{equation}

As for $(ii)$, considering $I$ and $J$ as constant families of
Fredholm operators,  $\Ind\,I=0,\,
\Ind\,J=-[\Theta(\R^N)].$  Using  the logarithmic and   direct sum properties of
the index bundle together with  \eqref{dirsum}, we obtain
\begin{equation*}
\Ind L = [E^u( -\infty)] +[E^s( +\infty)]-[\Theta(\R^N)]=[E^s(
+\infty)]-[E^s( -\infty)],
\end{equation*}
which proves $(ii).$
\end{proof}

\begin{remark} \label{ker} Notice that from  \eqref{ker+}, \eqref{ker-} and \eqref{operator}  it follows that
in the case  of systems of the special  form \eqref{const}
elements  of  $\Ker L_{\la}$ are sequences $(x_n)\in X$ such that
$x_0\in E^s(\la,+\infty) \cap E^u(\la,-\infty)$
and
$$x_n=a(\la,+\infty)^{n}x_0,  \hfor  n \geq 0   \hand  x_n=a(\la,-\infty)^{n}x_0,  \hfor
n \leq 0.$$
\end{remark}

\section{Parity and topological degree  of $C^1$-Fredholm maps}

In order to deal with the nonlinear aspects of the problem we will use an extension of the well known Leray-Schauder
degree to  proper Fredholm maps of index $0$ introduced in  \cite{Pej-Rab} under the name of {\it base point degree.}
This construction uses a homotopy invariant of paths of Fredholm operators of index $0$  called {\it parity}  which is
closely related to the index bundle.  We will briefly review  the concept of parity  and the construction of the  base
point degree. We are specially interested in the particular form of the homotopy property of the base point degree since
it represents  the main argument  in our proof of Theorem
\ref{theorem}.

From now on  we will consider only Fredholm operators of index $0$.

Given a continuous map $L\colon [a,b]\ra \Phi_0(X,Y),$ a
 {\it regular parametrix} (or regularizator)  for
the path $L$ is a path of isomorphisms $P\colon [a,b]\ra\, Iso(Y,X)$ such that $ L_t P_t = \Id_Y -K_t$  and $P_tL_t =\Id_X -K'_t,$
where  $K_t,K'_t$ are operators of finite rank.

Every path in $\Phi_0(X,Y)$ possesses at least one parametrix. Below  we describe  a construction related to the index bundle
(see  \cite{Fi-Pej-88} for details):

Given $L\colon [a,b]\ra \Phi_0(X,Y)$, arguing
as in the construction of the index bundle (see Section $3$), we
take a finite dimensional subspace $V$ of $Y$ and consider the
vector bundle
\begin{equation}
E= \bigcup_{t \in [a,b]}\, \{t\} \times L_{t}^{-1}(V).
\end{equation}
It is easy to see that  $\dim E_t =\dim V$, where
$E_t:=L_{t}^{-1}(V).$ Since $E$ is a trivial bundle  there is a
vector bundle isomorphism $T\colon E\ra \Theta(V) = [a,b] \times
V.$ Let $Q_t$ be a family of projectors of $X$ with $\text{Im }
Q_t=E_t,$ let $Q'$ be a projector with $\Ker\;Q'=V$ and let $A_t =
Q'L_t + T_tQ_t$. It is easy to see that $A_t$ is an isomorphism
for any $t\in [a,b].$ Its inverse $P_t:=A^{-1}_t$ is  a {\it
regular parametrix} for  $L$ because, as it is easy to see, $ L_t
P_t = \Id_Y -K_t$ with $\text{Im } K_t \subset V$ and $P_tL_t
=\Id_X -K'_t$ with $\text{Im } K'_t \subset E_t.$

 Let now $L\colon [a,b]\ra \Phi_0(X,Y)$ be a path such that both $L_a$ and $L_b$ are invertible operators. Let $P$ be a parametrix for $L.$  Then $ L_t P_t = \Id_Y
-K_t$  is invertible
for $t=a$ and $t=b$, and so are  its restrictions  $C_t \colon V \ra V$  to any finite dimensional subspace V  containing the images of $K_t.$

The {\sl parity} of the path $L$ is the element $\sigma(L) \in
\Z_{2}= \{1,-1\}$ defined  by
\begin{equation*}
\sigma(L) = \sign\det C(a)\sign \det C(b).
\end{equation*}
It is easy to see that this definition is independent of the
choices involved and that the parity is invariant under homotopies
of paths with invertible end points. Moreover it has the following
multiplicative property:  if $\{I_k, 1\leq k\leq m\}$  is a
partition of $I=[a,b]$ then \begin{equation} \label{mult}
\sigma(L) = \prod_{k=1}^m  \sigma(L_{\mid{I_k}}). \end{equation}   It can be
shown  that
$\sigma(L)= 1$ if and only if $L$ can be deformed to a
family of invertible operators by a homotopy  which keeps the  end
points invertible (see  \cite{Fi-Pej-88}).

If the path $L$  is closed, i.e.,    $L_a=L_b,$ then, via the identification $S^1\simeq [a,b]/\{a,b\}$   we can consider the path $L $ as a map $L
\colon S^1\rightarrow \Phi_{0} (X,Y)$  and  relate  the
parity  of a closed path with the obstruction to triviality  $w_1\colon
\widetilde{KO}(S^1) \ra \Z_2.$

\begin{lemma}\label{parity-lemma}
Under the above assumptions,
\begin{equation}\label{paritywhitney}
\sigma(L) = w_1(\Ind L).
\end{equation}
\end{lemma}

\begin{proof}  Since, by Proposition \ref{isomorphism-w}, $w_1$ is an
isomorphism  of $\widetilde{KO}(S^1)$ with $\Z_2$ it is enough to
check that $\sigma(L)= 1$ if and only if $\Ind L =0.$ Let us
recall that two bundles are {\it stably equivalent} if  they
become isomorphic after  addition of trivial bundles on both
sides. It is well known \cite{Hus} that stable equivalence classes
form a group isomorphic to the reduced Grothendieck group
$\widetilde{KO}(\Lambda ).$

Since the index bundle of a family of Fredholm  operators of index
$0$ belongs to $\widetilde{KO}(\Lambda),$
it follows that $\Ind L$
can be identified with the stable equivalence class of the vector
bundle $E=\bigcup_{\la \in \Lambda} \, \{\la\} \times
L_{\la}^{-1}(V)$  arising in the construction \eqref{1.1}.

If $\Ind L =0 \in \widetilde{KO}(S^1)$, then, for some $k \geq 0,$
$E\oplus \Theta(\R^k)$ is isomorphic to the trivial bundle
$\Theta(V\oplus \R^k),$ where $V$ is as in \eqref{1.1}.  Taking in
the definition of the index bundle in  \eqref{defind}  a larger
subspace $V'$ such that $V'/V\cong \R^k$  we can assume that $E$
itself  is trivial.  If we use such a $V'$ in the construction of
a parametrix  for $L \colon [a,b]\ra \Phi_0(X,Y)$ as described above, then we get
 $\sigma(L)=1.$

On the other hand, if $\sigma(L)= 1$, then one can modify any
parametrix of $L$ on   $[a,b] $ to a parametrix $P$ with
$P_a=P_b,$ which defines $P$ on $S^1.$  Then for  any $t\in S^1$
we have $P_tL_t = \Id _Y -K_t$ with $K_t$  compact and  therefore
$$H(t,s)= P^{-1}_t (\Id _Y - sK_t) $$ is a homotopy in $\Phi_0(X,Y)$ between $L$ and a family of isomorphisms, which implies that $\Ind L=0.$ \end{proof}
\sk
Now let us sketch  the construction of the  base point degree in \cite{Pej-Rab}.

Let $\mathcal{O}\subset X$ be an open simply connected set and let
$f\colon \mathcal{O}\ra Y$ be a $C^1$-Fredholm map of index $0$
that is proper on closed bounded subsets of the domain (recall
that a $C^1$-map $f\colon \mathcal{O}\to Y$ is Fredholm of index
$0$ if the Fr\'{e}chet derivative $Df(x)$ of $f$ at $x$ is a
Fredholm operator of index $0$, for all $x\in \mathcal{O}$). Using
the parity we can assign to each regular point (\footnote{ $p$ is a
regular point of $f$ if  $Df(p)$ is an isomorphism.
})  of the map $f$
an orientation $\epsilon(x)=\pm 1$ with similar  properties to the
sign of the Jacobian determinant in finite dimensions. For this we
choose a fixed regular point $b$ of $f,$ called {\it base point},
and then the corresponding orientation $\epsilon_{b}(x)$ at any
regular point $x$ is uniquely defined by the requirement
$\epsilon_{b}(x)=\sigma(Df\circ\gamma)$, where $\gamma$ is any
path in $\mathcal{O}$ joining $b$ to $x$. Since $\mathcal{O}$ is
simply connected, the independence from the choice of the path
follows from the homotopy invariance of the parity.

Let $\Omega $ be an open bounded set whose closure is contained in
$\mathcal{O}$  such that
$0$ is a regular value of the restriction of $f$ to $\Omega$ and such that $0\not\in f(\partial\Omega).$  Then
the base point degree of $f$ in $\Omega$ is defined by
\begin{equation}\label{deg}
\text{deg}_b(f,\Omega,0)=\sum_{x\in f^{-1}(0)}\epsilon_b(x).
\end{equation}
In the above definition  we use the convention that a sum over the empty set is $0.$

It is proved in \cite{Pej-Rab} that this assignment extends to an
integral-valued degree theory for $C^1$-Fredholm maps defined on
simply connected sets that  are proper on closed bounded
subsets of its domain. The base point degree is invariant under
homotopies only up to sign and, as a matter of fact, since the identity map of a
(separable) Hilbert space can be connected to an isomorphism of
the form the identity map plus a compact map whose Leray-Schauder degree is $-1,$ no degree theory for general Fredholm maps extending the Leray-Schauder
degree can be homotopy invariant.

The main reason for introducing the base point degree is  that  the change in sign along a homotopy can be determined using the parity.

An admissible  homotopy in our  setting is a continuous family of
$C^1$-Fredholm maps   $h\colon [0,1]\times\mathcal{O}\ra Y$
parametrized by $[0,1]$  which is proper on closed bounded subsets
of $[0,1]\times\mathcal{O}.$   As usual, continuous family of
$C^1$-maps means that $h$ is continuous, differentiable in  the
second variable with the derivative continuously depending on
$(t,x).$

Our proof of the main result  will be based on the following  \emph{ homotopy variation property} of the base point degree (see \cite[Lemma 2.3.1]{Pejs-1}):

\begin{lemma}\label{varhomotopy}

Let $h\colon [0,1]\times\mathcal{O}\ra Y$  be an admissible
homotopy, and let  $\Omega$ be  an open bounded subset of X such
that $0\not\in h([0, 1]\times\partial\Omega).$ If  $b_i\in
\mathcal O$  is a  base point for  $h_i =: h(i,-); i= 0,1, $ then
\be\label{homotopy}
\text{\emph{deg}}_{b_0}(h_0,\Omega,0)=\sigma(M)\text{\emph{deg}}_{b_1}
(h_1,\Omega,0),\ee where $M\colon [0,1] \ra \Phi_0(X,Y)$ is the
path $ L\circ \gamma,$ where $L(t,x)=Dh_t(x)$ and  $\gamma$ is any
path joining $(0,b_0)$ to $(1,b_1)$ in $[0,1]\times \mathcal{O}.$
\end{lemma}

\section{Proof of Theorem \ref{theorem}}

Let  $\ff\colon  \Z\times  S^1\times \R^N\ra \R^N$ be a continuous
family of  nonautonomous dynamical systems verifying
$(A0)$--$(A2).$ Take  $X:={\bf c}(\R^N)$ and let  $F\colon
S^1\times X\ra X$ be defined by
\begin{equation}\label{nem}
F(\la,\xx)=(f_n(\la,x_n)),\; \text{for} \;\xx\in X
\;\text{and}\;\la\in S^1.
\end{equation}

Firstly we observe  that $F\colon S^1\times X\ra X$ is well
defined. Indeed, given $\xx\in X$, taking into account Assumption
$(A2),$ we deduce that
\begin{equation*} C_{\la}:=\sup_{(n,s)\in \Z\times [0,1]}\Bigg\|\frac{\partial
f_n}{\partial x}(\la,sx_n)\Bigg\|<\infty, \end{equation*} for all
$\la\in S^1$. Hence using the mean value estimate we get
\begin{equation*}
\|f_n(\la,x_n)\|=\|f_n(\la,x_n)-f_n(\la,0)\|\leq\sup_{s\in
[0,1]}\Bigg\|\frac{\partial f_n}{\partial x}(\la,sx_n)\Bigg\|\cdot
\|x_n\|\leq C_{\la}\|x_n\|.
\end{equation*}
Thus $f_n(\la,x_n)\ra 0$ as $n\ra\pm\infty$, which proves that the
map $F\colon S^1\times X\ra X$ is well defined. Furthermore, the
same argument allows us to define  the family of linear bounded
operators  $T\colon S^1\times X\ra \mathcal{L}(X,X)$ by
\begin{equation}\label{frechet}
T(\la,\xx)\yy:=\Bigg(\frac{\partial f_n(\la,x_n)}{\partial
x}y_n\Bigg), \end{equation}  for $\xx=(x_n), \yy=(y_n)\in X$ and $\la\in S^1$.

\begin{lemma} \label{fredmap}
\indent
\begin{itemize}
\item[i)] If \emph{$\ff$}
verifies $(A1)$ and $(A2),$  then the map  $F\colon S^1 \times
X\ra X$ defined by \eqref{nem} is a continuous family of
$C^1$-maps parametrized by $S^1.$  Moreover
$DF_\la(\emph{\xx})=T(\la,\emph{\xx})$.

\item[ii)] If also  $(A3)$  holds, then there exists a  closed neighborhood $D=\bar B(\emph{\oo},\delta)$ of
\emph{$\oo$} in $X$ such that the restriction of  $G:=S-F$
to $S^1\times D$ is a proper continuous family of $C^1$-Fredholm
maps of index $0$.  Namely,  $G\colon S^1\times D \ra X$ is continuous and proper. Moreover, for any $\la \in S^1,$ the map
$G\colon S^1\times D \ra X$
is  differentiable in the second variable and $DG_\la(\emph{\xx})$ is a Fredholm operator of index $0$ continuously depending on
$(\la,\emph{\xx}).$
\end{itemize}
  \end{lemma}

\begin{proof}

The proof of $i)$  follows the  lines of \cite[Lemma 2.3]{Potz-2}.
We sketch it below  for convenience of the reader,
since our
setting is slightly more general than the one   in
\cite{Potz-2}.
Notice that $(A1)$ tells that the sequence,
for $j=0,1,$
$\frac{\partial^j f_n}{\partial x^j}$
is uniformly
equicontinuous, while $(A2)$ means that the restriction of the
same sequence
to bounded subsets of the domain is equibounded.

 The continuity of  $F$ and the map $T$ defined above follows easily from the equicontinuity assumption $(A1).$

For fixed $\xx\in X$ and $\la\in S^1$  we will show that
$DF_\la(\xx)=T(\la,\xx)$.  To this
end, let \be R(\xx,{\bf h};\la):=\|F(\la,\xx+{\bf
h})-F(\la,\xx)-T(\la,\xx){\bf h}\|,\ee where ${\bf h}\in {\bf
c}(\R^N)$ and $\la\in S^1$. We are to show that $\frac{R(\xx,{\bf
h};\la)}{\|{\bf h}\|}\ra 0$ as $\|{\bf h}\|\ra 0.$ Let
$$
c_n({\bf h};\la):=\sup_{s\in[0,1]}\Biggl\|\frac{\partial
f_n(\la,x_n+sh_n)}{\partial x}-\frac{\partial
f_n(\la,x_n)}{\partial x }\Biggr\|,
$$
for $n\in\Z$. Then Assumptions  $(A2)$ and $(A1)$
imply that \be c_n({\bf h};\la)<\infty \text{ and }\sup_{n\in\Z}
c_n({\bf h};\la)\ra 0 \text{ as } \|{\bf h}\|\ra 0. \ee Then
\begin{align*}
&\Biggl\|f_n(\la,x_n+h_n)-f_n(\la,x_n)-\frac{\partial f_n(\la,x_n)}{\partial x}h_n\Biggr\|=\\
&\Biggl\|\int_0^1
\frac{\partial f_n(\la,x_n+sh_n)}{\partial x}h_n ds-\frac{\partial f_n(\la,x_n)}{\partial x}h_n\Biggr\|\leq\\
&\int_0^1 \Biggl\|\frac{\partial f_n(\la,x_n+sh_n)}{\partial x
}-\frac{\partial f_n(\la,x_n)}{\partial x}\Biggr\|ds\cdot
\|h_n\|\leq\\
&\int_0^1\sup_{n\in\Z}c_n({\bf h};\la) ds\cdot \|h_n\|=
\|h_n\|\cdot\sup_{n\in\Z}
c_n({\bf h};\la)
\leq \|{\bf h}\|\cdot
\sup_{n\in\Z}c_n({\bf h};\la).
\end{align*}
Hence \be 0\leq R(\xx,{\bf h};\la)\leq\|{\bf
h}\|\sup_{n\in\Z}c_n({\bf h};\la), \ee which implies that
$\frac{R(\xx,{\bf h};\la)}{\|{\bf h}\|}\ra 0$ as $\|{\bf h}\|\ra
0$. This completes the proof of $i)$ since we already know that $T$ is continuous.

Let us prove  $ii).$  By the previous considerations  the map
$G(\la,\xx)=S\xx-F(\la,\xx)$ is a continuous family of $C^1$-maps.
Since $a_n(\la)=\frac{\partial f_n}{\partial x}(\la,0),$ it
follows from \eqref{frechet}  that  $DG_\la(\oo)$ is the operator
$L_{\la}\colon X\ra X$ defined by
\begin{equation}\label{linear} L_{\la}\xx=(x_{n+1}-a_n(\la)x_n).
\end{equation} Being $\aa$ asymptotically hyperbolic,  by Theorem
\ref{prop:ind},  the operator   $L_\la$ is Fredholm with  index
given by   \eqref{dim}.  Thus $\ind L_\la=0,$ since by $(A3)$ the
stable subspaces at $\pm \infty$ have the same dimension.

Since $\Phi_0(X,X)$ is  an open subset  of $\mathcal{L}(X,X),$  by
continuity of $DG_\la(\xx)$ and compactness of $S^1,$  there
exists  a $\delta>0$ such that the restriction of $G$ to
$S^1\times B(\oo,\delta)$ is a continuous family of $C^1$-Fredholm
maps of index $0$. Using compactness of $S^1$ again,  we can  find
an eventually  smaller  $\delta $ such that for all $\la \in S^1$
the restriction  of $G$ to $S^1\times \bar B(\oo,\delta)$ becomes
proper, because continuous families of $C^1$-Fredholm maps are
locally proper (\cite[Lemma 3.5]{Be-Fu}).
\end{proof}

Now we can  finalize the proof of Theorem \ref{theorem} using the homotopy variance property of  base point degree.
By  $(A4), \,L_{\la_0}$  is injective, and hence  it
 follows from the Fredholm alternative  that   $L_{\la_0}$ must be  invertible.
 The inverse function  theorem implies that for $\delta>0$ small enough $\oo$ is the only
solution of $G_{\la_0}(\xx)=\oo$ in $B(\oo,\delta)$.   Moreover we
can take $\delta$ so small  that $G\colon S^1\times
\bar{B}(\oo,\delta)\ra X$ verifies $ii) $ of the above Lemma. To
simplify notation we suppose that $\la_0=1\in S^1.$

Assume that for  $\varepsilon<\delta$ there are no homoclinic
solutions $(\la,\xx)$  of  \eqref{main-system}  with
 $\|\xx\|=\varepsilon,$  then $G_\la(\xx) \neq \oo$  on $\partial B(\oo,\varepsilon).$  Consider the  homotopy
$H\colon [0,1]\times \bar B(\oo,\varepsilon)\ra X$ defined by
$H(t,\xx)=G(\text{exp}(2\pi it),\xx).$ Then $H$ is an admissible
homotopy with  $H_0 =G_1=H_1$. Furthermore, we can take $b=\oo$ as
the base point for both $H_0$ and $H_1$. Since $\oo\not\in
H\big([0, 1]\times\partial B(\oo,\varepsilon)\big)$ by Lemma
\ref{varhomotopy}, with  by $\gamma (t)=(t,\oo),$
\begin{equation*}
\text{deg}_{\oo}(H_1,B(\oo,\varepsilon),\oo) =\sigma(M)
\text{deg}_{\oo}(H_0,B(\oo,\varepsilon),\oo),
\end{equation*}
where $M(t) =L_{\text{exp}(2\pi it)}.$ By definition of the degree
for a regular value \eqref{deg} we have \be
\text{deg}_{\oo}(H_j,B(\oo,\varepsilon),\oo)=\text{deg}_{\oo}(G_1,B(\oo,\varepsilon),\oo)
=1,\, j=0,1.\ee

Thus $\sigma(M)=1.$   But $\sigma(M)$ coincides  with the parity of the closed path $L.$ Hence,  by Lemma \ref{parity-lemma}, Theorem \ref{prop:ind} and \eqref{morewhit}
$$ 1=\sigma(L) =w_1(\Ind\,L) = w_1(E^s (+ \infty)) w_1(E^s(-\infty)),$$ which
contradicts our assumption.
\qed

\section{An example}

In this section we are going to illustrate the content of Theorem \ref{theorem}
 comparing our result  with the standard theory.
\vskip5pt

For $\la=\exp(i\theta)$,
$0\leq \theta\leq 2\pi$, we put
\begin{equation*}
a(\la)=a(\exp i\theta ):=\begin{pmatrix}
  1/2+(3/2)\sin^2 \theta/2 & -(3/4)\sin \theta \\
  -(3/4)\sin \theta & 1/2+(3/2)\cos^2 \theta/2
\end{pmatrix}
\end{equation*}
and consider  the  linear nonautonomous system
${\aa} =(a_n (\la)) \colon \Z\times S^1 \ra GL(2)$ defined by
\begin{equation}\label{constant}
a_n(\la) =
\begin{cases} a(\la)& \hif\; n\geq 0,
 \\ a(1)
 &\hif\; n<0.
\end{cases}
\end{equation}

Notice that system $\aa$ has the special  "jump" form \eqref{const}, used in the proof of Theorem \ref{prop:ind}.

 Since independently of  $\la \in S^1$ the matrix $a(\la)$ has   two eigenvalues  $1/2$ and $2,$ the   system $\aa$ is asymptotically  hyperbolic.

We will  apply our results to nonlinear perturbations of $\aa.$ We
compute the asymptotic stable bundles of $\aa$  at
$\pm\infty\!:$
\begin{align*}
E^s(+\infty)&=\{(\la,u)\in S^1\times \R^2\mid
u=t(\cos(\theta/2),\sin(\theta/2)),
\la=\exp(i\theta),t\in\R\},\\
E^s(-\infty)&=\{(\la,u)\in S^1\times \R^2\mid u=(t,0),t\in\R\}.
\end{align*}
Thus $E^s(-\infty)$ is a trivial bundle and hence
$w_1(E^s(-\infty))=1.$ In order to compute  $w_1(E^s(+\infty))$ we
notice that $v_\theta = (\cos(\theta/2),\sin(\theta/2))$ is a
basis for $ E^s_\theta(+\infty)$ which is the fiber of the
pullback $E' $ of $ E^s(+\infty)$ by the map $p\colon [0,2\pi] \ra
S^1$ defined by  $p(\theta) = \exp(i \theta).$

Since $v_0=(1,0)$ and $v_{2\pi}=(-1,0),$ the determinant of the
matrix $C$ arising in \eqref{whitney} is $-1.$ Hence
$w_1(E^s(+\infty))=-1 \neq w_1(E^s(-\infty)).$ Notice that
$E^s(+\infty)$ is an infinite  Moebius band while $E^s(-\infty)$
is an infinite cylinder.

If $\hh$ is any nonlinear perturbation of  $\aa$  verifying
$(A0)$--$(A2)$ and $(A3'),(A4')$ then by Corollary \ref{cor2} the
family $\ff =\aa +\hh$ must have nontrivial homoclinic solutions
bifurcating from the stationary branch at some $\la_* \in S^1.$

On the other hand, let us consider the family  $L$ of  operators
$L_\la$ defined by
\begin{equation*}
L_\la(\xx) (n)=
\begin{cases}
x_{n+1}-a(\la)x_n & \hif\;  n\geq 0, \\
x_{n+1}-a(1)x_n & \hif\; n<0.
\end{cases}
\end{equation*}

By Remark \ref{ker} $\Ker L_\la$ is isomorphic to
$E^s(\la,+\infty)\cap E^u(\la,-\infty).$

But $E^u(-\infty)=\{(\la,u)\in S^1\times \R^2\mid
u=(0,t),t\in\R\}$ and hence a nontrivial intersection  arises only
for $\theta =\pi,$ i.e.,  $ \la =-1.$
Thus  $\Ker L_\la \neq 0$ only if $\la=-1.$

Now, if  $\hh$ verifies, in a neighborhood of  $\la=-1,$ that $\frac {\|h_n(\la,x)\|}{\|x\|} \ra 0$ uniformly in $(n,\la),$ we can use the classical approach based on Lyapunov-Schmidt reduction  in order to obtain  the existence of a branch of homoclinics
bifurcating  from the stationary branch at this point.

Indeed, if the above  condition holds true, the family
 $H \colon S^1\times \cc(\R^2) \to \cc(\R^2),$
induced on function
spaces, verifies $H(\la,\xx)=o(\|\xx\|)$ as $\xx\ra \oo.$ Being
$\Ker L _{-1}$ one dimensional,  we can check the hypothesis
 of  the Crandal-Rabinowitz Bifurcation Theorem in order to find  at $\la =-1$  a bifurcating   branch of homoclinics  \cite{Cr-Ra}.

Instead, by using  Corollary \ref{cor2},  we lost any information about the position of the bifurcation point,  but we  proved  the appearance  of nontrivial homoclinic trajectories for a  rather general class of perturbations.

To some extent,  the use in  bifurcation theory  of elliptic invariants "at large" in place of the Lyapunov-Schmidt  method  parallels the use of the  topological degree instead of the multiplicity of an isolated solution in continuation problems.  This observation  is formulated more precisely in \cite{Pejs-1}.  Let us point out however,  that the relation between the birth of homoclinics  and the topology of asymptotic stable bundles  is interesting by itself and goes beyond the formal aspects of the abstract theory.

\section{Comments}

It should be noted that our results  can
be proved under weaker assumptions. Mainly, it suffices to
assume that the nonautonomous difference system admits an exponential dichotomy \cite{Ba, Potz, Potz-2}. This   becomes useful  in  dealing with difference  equations  in Banach spaces.
On the other hand our  method  can be easily adapted in order to study bifurcation of homoclinics  on manifolds. Following \cite{Ab-Ma2}, to each finite dimensional manifold $M$ and  discrete dynamical system $\ff$ on $M$  having  $x\in M$  as a stationary trajectory   we can associate the Banach manifold $\cc_x(M)$ which is a natural place for the study of  trajectories of the dynamical system $\ff$ homoclinic to $x.$

There are two  interesting problems   which  were not considered here. Namely, global bifurcation, which studies the existence of
connected branches  of solutions  and their  behavior, and the  existence of  large  homoclinic trajectories  using bifurcation from infinity.  These will be treated in a forthcoming paper of the present authors.

\medskip
Received xxxx 20xx; revised xxxx 20xx.
\medskip

\end{document}